\documentclass{ifcolog}

\usepackage{cleveref, fontawesome5}
\usepackage{makeidx, lmodern, doi}

\usepackage{fouche/lipics-fouche}

\newtheorem{notation}[subsection]{Notation}
\newtheorem{remark}[subsection]{Remark}
\newtheorem{definition}[subsection]{Definition}
\newtheorem{proposition}[subsection]{Proposition}
\newtheorem{lemma}[subsection]{Lemma}
\newtheorem{theorem}[subsection]{Theorem}
\newtheorem{assumption}[subsection]{Assumption}

\usepackage{hyperref}
\hypersetup{ hyperfootnotes=true
           , bookmarksnumbered=true
           , bookmarksopen=true
           , breaklinks=true
           , urlcolor=red!40!black
           , colorlinks%
           , citecolor=blue!40!black%
           , linkcolor=blue!40!black%
           }

\usepackage{moirai}

\tikzset{lbl/.style={font=\tiny,fill=white, inner sep=1pt, draw=gray, ultra thin}}
\tikzset{nd/.style={}}

\let\cate\mathsf
\def\Opt{\mathsf{Optic}}
\DeclareDocumentCommand{\esc}{O{magenta} O{blue} m m m O{M}}{
  \left(
	  \begin{tikzpicture}[baseline=(current bounding box.center)]
	    \arTwo[#1!10]{{#4}_1}\gau{$#3$}\down{\dro{$#5$}}\up{\dro{$#6$}}
	  \end{tikzpicture}\, , \,
	  \begin{tikzpicture}[baseline=(current bounding box.center)]
	    \twoAr[#2!10]{{#4}_2}\dro{$#5$}\down{\gau{$#3$}}\up{\gau{$#6$}}
	  \end{tikzpicture}
  \right)
}
\usepackage{xspace}
\def\A{\texttt{Aki}\xspace}
\def\B{\texttt{Bogdan}\xspace}
\def\C{\texttt{Candice}\xspace}

\def\aki{\includegraphics[scale=.1]{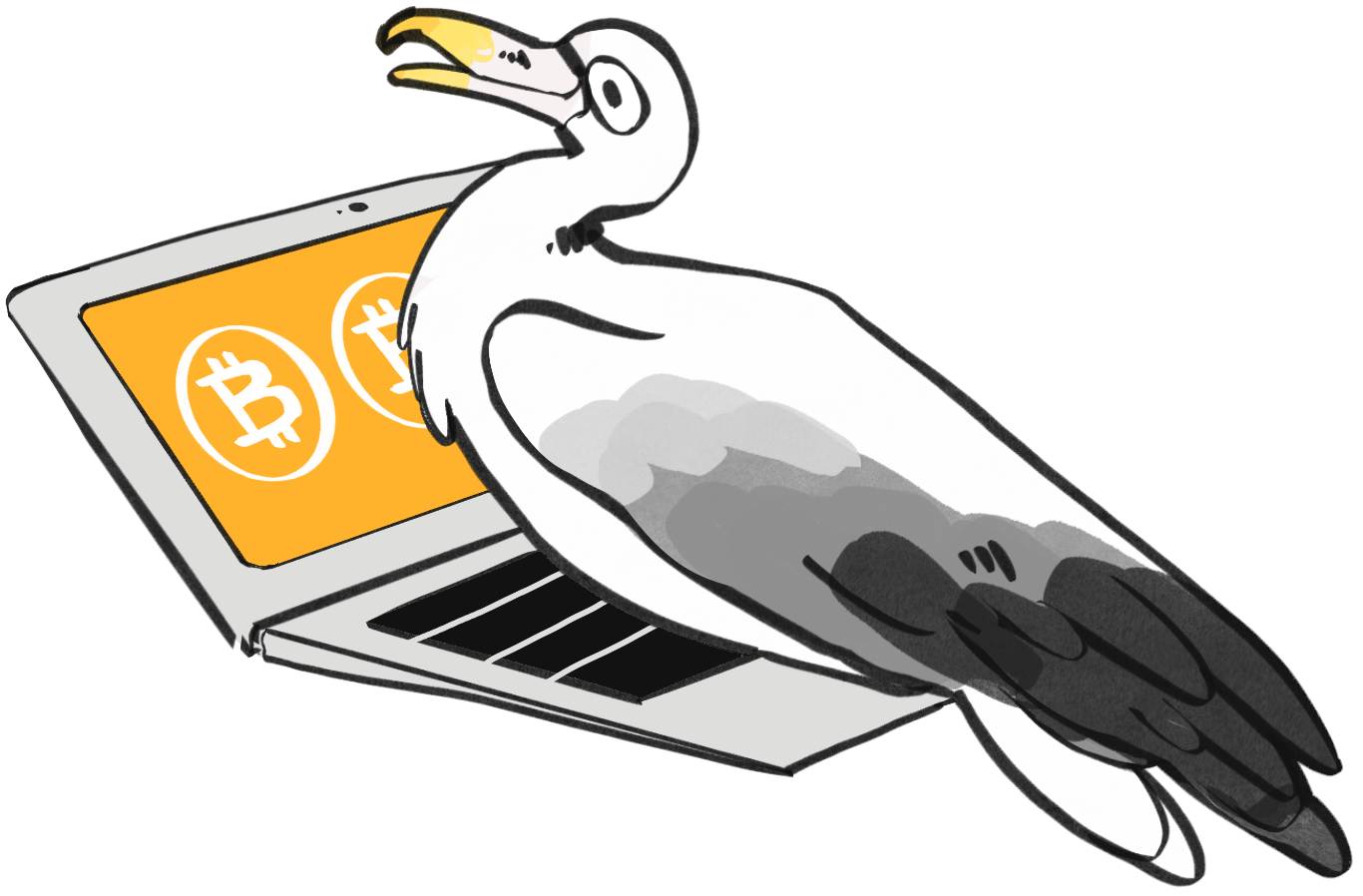}}
\def\bogdan{\includegraphics[scale=.2]{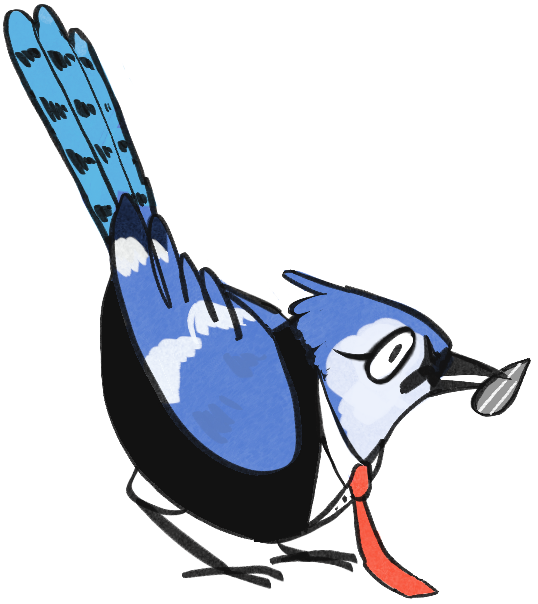}}
\def\candice{\includegraphics[scale=.1]{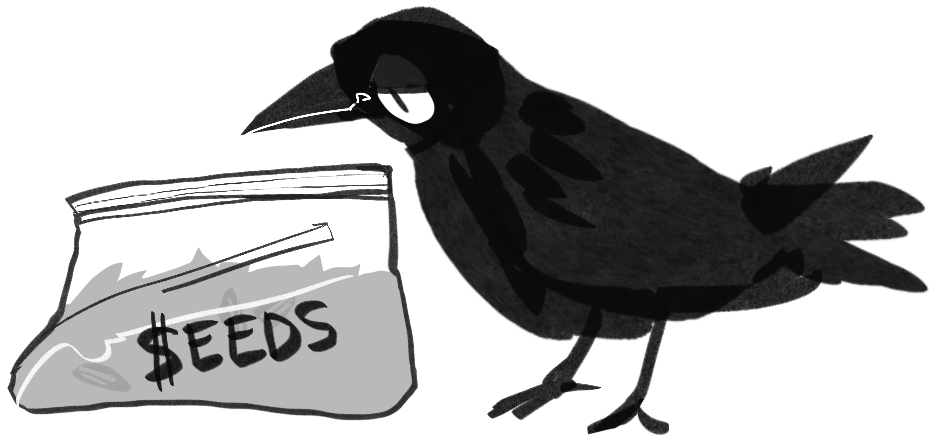}}
\def\vault{\includegraphics[scale=.1]{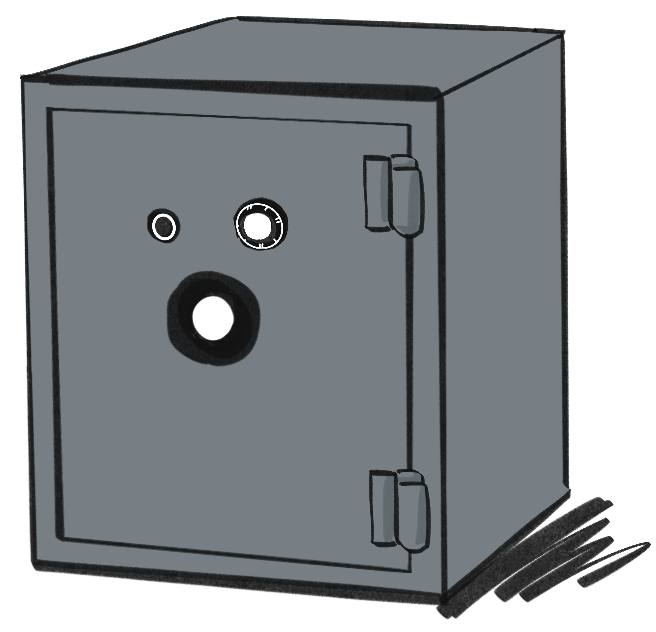}}

\usepackage{cjhebrew}

\def\lr#1#2{\langle #1,#2\rangle}
\titlethanks{All authors thank P.E.E. Gaviria and \emph{Dread Pirate Roberts}; the pictures of Aki, Bogdan and Candice have been kindly provided by \href{https://www.instagram.com/po.ppo/}{\faInstagram\ po.ppo}.}

\title{Escrows are optics}
\titlerunning{Escrows are optics}

\addauthor[fabrizio.romano.genovese@gmail.com]
  {Fabrizio Genovese\thanks{\href{https://gitcoin.co/grants/1086/independent-ethvestigator-program}{Independent Ethvestigator Program}}}
	{20squares}
\addauthor[fosco.loregian@gmail.com]
  {Fosco Loregian\thanks{ESF funded Estonian IT Academy research measure (project 2014-2020.4.05.19-0001)}}
	{Tallinn University of Technology, Estonia}

\addauthor[danielepalombi@protonmail.com]
  {Daniele Palombi}
	{20squares}

\authorrunning{Genovese, Loregian and Palombi}

\begin{document}
\maketitle

\begin{abstract}
  We provide a categorical interpretation for \emph{escrows}, i.e. trading protocols in trustless environment, where the exchange between two agents is mediated by a third party where the buyer locks the money until they receive the goods they want from the seller.

	A simplified escrow system can be modeled as a certain kind of morphism in the category of \emph{optics} on a monoidal category. When objects in the base category have monoid and comonoid structures, more involved kinds of escrows `with intermediaries' can be modelled as morphisms with action-like properties.
\end{abstract}

\section{Introduction}
\setlength{\epigraphwidth}{0.55\textwidth}
\epigraph{%
You have to be trusted by the people that you lie to\\
So that when they turn their backs on you,\\
You'll get the chance to put the knife in.}{Waters \& Gilmour}
In their most basic form, escrows can be described as a way to enforce mutual destruction in a trustless transactional environment.

Let us unpack this definition. In our work, a \emph{transaction} will mean a transfer of goods, whereas a \emph{trade} will be a mutual exchange of goods defined by transactions. For example, the transactions `\A pays money to \B' and `\B gives apples to \A' will define a (familiar) trade where \A buys some apples from \B.\footnote{Our paper employs other characters than Alice, Bob and Charlie: \emph{Aki} the albatross, \emph{Bogdan} the bluejay, and \emph{Candice} the crow.}

The structure of a simple trade is depicted below, with the downward-pointing arrows representing the flow of time.
\begin{figure}[ht]
  \begin{center}
    \begin{tikzpicture}[>=stealth']
      \draw[->] (0,0) -- (0,-4);
      \node[nd] (A) at (0,-1) {\aki};
      \draw[->] (3,0) -- (3,-4);
      \begin{scope}[yshift=-1cm]
      \node[nd] (B) at (3,-1.25) {\candice};
      \node[nd] (A') at (0,-2) {\aki};
      \draw[->, thick, gray] (A) -- node[font=\tiny,fill=white, inner sep=1pt, draw=gray, ultra thin] {send money} (B);
      \draw [->, thick, gray] (B) -- node[font=\tiny,fill=white, inner sep=1pt, draw=gray, ultra thin] {send goods} (A');
      \end{scope}
    \end{tikzpicture}
  \end{center}
  \caption{A simple trade.}\label{fig: simple trade}
\end{figure}
Now, an environment for a transaction is `standard' if trades happen in-person and in a situation where at least some basic form of law enforcement exists. In a legally standard environment, all parties in a trade can enforce exchange on the spot: if \A pays \B, but \B refuses to give \A apples, \A can go to the police, refuse to leave the shop until a refund or the goods are received, or hit \B on the head, depending on the jurisdiction in which they are. The same argument holds for \B, which allows keeping the basic structure of trades quite simple in most situations. In an economy of scale, \A and \B can transact over long distances by relying on a network of mutual bonds of trusts defining a banking system.

Things change radically in a trustless environment, where only part of the trade can be enforced. The most typical example is internet commerce, and in particular blockchains, for the following reasons:
\begin{itemize}
  \item In some cases, identities are pseudonymous or fully anonymous. Even when they are not, the parties taking part in a trade often live in different jurisdictions, making the application of law enforcement difficult or impossible.
  \item In some cases, the nature of the goods being traded may make it difficult or outright impossible to enforce a trade by legal means. The obvious example is people buying illegal goods (drugs or other prohibited chemicals) or services on marketplaces such as SilkRoad \cite{wiki:Silk_Road}. Here, even in a quarrel, no parties will ask the police for help lest they be incarcerated for far bigger crimes.
  \item In some cases, only part of the trade can be enforced and accounted for. This problem is particularly evident in a blockchain setting: The transaction `\A gives some bitcoins to \B' is publicly verifiable on the Bitcoin blockchain, and no dispute can arise in this respect. On the contrary, there is no easy way to account for `\B shipped apples to \A': we will have to rely on \B's word.
\end{itemize}
In such a setting, since we cannot suppose in any way that \A and \B know or trust each other, we see that a massive problem arises with trades as we usually conceive them: If \A pays \B, \B can run away without ever shipping the goods. On the contrary, if \B ships the goods first, \A can receive them and never pay \B back. More importantly, in a trustless paying system such as Bitcoin, there is no way for \A to claim money back in case of scams.

Escrows are a way to fix this problem. In their most basic form, they work as follows:
\begin{itemize}
  \item To buy some goods from \B, \A puts a sum of money totalling the price of the goods bought in an \emph{escrow system}. This method consists of a vault where money can go in but not out. \A, and only \A, has the power to release the money at a later stage. When doing so, the money is automatically sent to \B;
  \item \B can publicly check that the money is in the escrow vault. \B can be sure that the only way the money can leave the escrow is by being sent to him.
  \item \B ships the goods to \A;
  \item Upon receiving them, \A can release her escrow so that \B is paid.
\end{itemize}
This system is far from perfect, and various refinements of it have been proposed.

Still, even in this basic form, we can appreciate how escrows turn the strategy `running away with the money/goods' into a non-rational choice. In refusing to ship the goods, \B doesn't gain any money because \A will never release the escrow. Similarly, in refusing to pay \B, \A will never get her money back. So, both \A and \B have no economic incentive in not upholding their side of the trade and may do so only if motivated by irrational motives such as pure malicious intent.\footnote{Whoever is familiar with the blockchain ecosystem --and with life over the Internet-- knows that this is still a big problem. In practice, the introduction of time windows that release the escrow to the original owner if certain conditions are met can mitigate the problem. For the sake of simplicity, we will not focus on this in this work.} The basic structure of an escrow is depicted in \autoref{fig: escrow trade}.
\begin{figure}
  \begin{center}
    \begin{tikzpicture}[yscale=1.5,>=stealth']
      \draw[->] (0,0) -- (0,-3);
      \node[nd] (A) at (0,-.5) {\aki};
      \node[nd] (A') at (0,-2.25) {\aki};
      \draw[->] (3,0) -- (3,-3);
      \node[nd] (E) at (3,-.75) {\vault};
      \node[nd] (E') at (3,-1.5) {\vault};
      \node[nd] (E'') at (3,-2.5) {\vault};
      \draw[->] (6,0) -- (6,-3);
      \node[nd] (B) at (6,-.5) {\candice};
      \node[nd] (B') at (6,-1.5) {\candice};
      \node[nd] (B'') at (6,-2.5) {\candice};
      \draw[->, thick, gray]
      (A) -- node[above,lbl] {lock money} (E);
      \draw[->, thick, gray] (B) -- node[lbl] {money locked?} (E');
      \draw[->, thick, gray] (E')    -- node[lbl] {Yes} (B');
      \draw[->, thick, gray] (B')    -- node[below, lbl] {send goods} (A');
      \draw[->, thick, gray] (A')    -- node[lbl] {unlock} (E'');
      \draw[->, thick, gray] (E'')    -- node[lbl] {send money} (B'');
    \end{tikzpicture}
  \end{center}
  \caption{An escrow trade.}\label{fig: escrow trade}
\end{figure}
This paper aims to describe escrows formally, with a particular focus on their compositional nature. For this, the language of category theory is fundamental. Our main contribution is showing that escrows are examples of an \emph{optic} and that thus they organise in a category whose composition law accounts for a tangible property of this kind of transactions.

In detail, in~\cref{sec: escrows are optics} we define a simplified escrow system as a certain kind of optic in a monoidal category $\clM$ (e.g., the category of sets with cartesian product); escrows can be regarded as morphisms of a category $\clE(\clM)$, with objects are the same of $\clM$, and where the hom-objects are $\lr X Y = \Opt_\clM(\cop Y X, \cop X Y)$. When $X$ is a comonoid and $Y$ is a monoid in $\clM$, $\clE(\clM)(X,Y)$ is a monoid in $\Set$, acting on the set of optics $\cop B B \to \cop XY$. Moreover, in~\cref{sec:usingAgentsAsIntermediaries} we will define a map
	\[\notag\xymatrix@R=0cm{\lhd : \lr YX \times  \Opt(\cop YX, \cop BB) \ar[r] & \Opt(\cop YX, \cop{X\otimes B}{Y\otimes B})}\]
having action-like properties. This has the following interpretation: the object $B$ acts as an intermediary in a transaction between $X$ and $Y$, modeled by an escrow in $\lr YX$.

\subsection{Related work on optics, and conventions}
Optics originated in functional programming as a compositional solution to the problem of accessing fields of nested data structures; currently, optics find application in diverse branches of applied category theory like categorical probability theory \cite{daf90452d38580f2e19794e616379995861ecea5,da2b5f4c85fe5c1aff024f6150c2e83575ac1c76} the theory of open games \cite{f9b982a8d01989ee54b4ab452d582be659fd560a,475daed80393f9c3a97e61081ebcc31c17977cf4} (see also the connection with lenses, particular optics: \cite{ecd7fdf53c994685194f3cb907eb68b2418110bd,e26e834e9f21cc2c8c7a598c0bbf25b3271c7637}; see \cite{3a47aaca77631f0db4e035239ece5d1829e7d511} for a categorical overview and generalisation) functional programming \cite{38c6ae5831d38d33e766ca236064fd075f8d49c0,pickering2017profunctor,10aecab386fc380275c87fce2372894bf771d748, noi}; see \cite{50efd1c56c11adca8b171aebcc0a41c02c8136d9} for a string-diagrammatic calculus (which is not the one we employ here; instead, we use \cite{276b25d17339ace4121eff16e569bddceebd70e6}).

The blanket assumption that we make throughout the entire paper is to work in a fixed symmetric monoidal category $\clM$. We will use the string diagram notation for monoidal categories whenever convenient. We freely employ basic facts about monoidal categories \cite{kelly}, the notion of actions of an internal monoid in a monoidal category, and elementary facts about `coend calculus' \cite{coend-calcu}, namely the fact that given a category $\clX$ and a functor $T : \clX^\op\times \clX \to \Set$ its `coend' $\int^X T(X,X)$ is the coequaliser \cite[2.4]{Bor1}
\[
  \xymatrix{
    \coprod_{Y\to X} T(X,Y) \ar@<.33em>[r]\ar@<-.33em>[r] & \coprod_{X\in \clX} T(X,X) \ar[r]  & \int^X T(X,X)
  }
\]
of a suitable pair of maps defined by the conjoint action of $T$ on morphisms (see \cite[(1.29)]{coend-calcu}).

Although indirectly, we work with the notion of bicategory of profunctors; the bicategory $\cate{Prof}$ has objects the small categories $\clA,\clB$, 1-cells from $\clA$ to $\clB$, denoted $\clA \pto \clB$, are the functors $P : \clA^\op\times\clB \to \Set$, and 2-cells between $P,Q : \clA\pto \clB$ are the morphisms in the category $\Cat(\clA^\op\times\clB,\Set)$. The 1-cells of $\cate{Prof}$ can be composed with a `matrix product' rule
\[
Q\odot P : \clA \pto \clC : (A,C)\mapsto \int^B P(A,B)\times Q(B,C)
\] if $\clP : \clA\pto\clB$ and $Q : \clB\pto \clC$.
\section{Escrows are optics}\label{sec: escrows are optics}

In the following, we interpret objects $A$ (\A), $B$ (\B), $\dots$ of $\clM$ as entities of some kind, transacting in money and goods. As already discussed, we assume the interaction between $A,B\in\clM$ is always non-spiteful, in the sense that we dismiss the possibility that $A$ (resp., $B$) wants to cause damage to $B$ (resp., $A$) even at the cost of losing money (resp., goods).

Morphisms $A \to B$ will be interpreted from time to time as exchanges of goods, money or information between $A$ and $B$. The nature of goods and information being transacted will often be kept implicit: $A \to B$ means that some information or good possessed by $A$ is given to $B$.

On one side, money and goods are \emph{resources}, and as such, they can be compounded (`if $A$ gives me an apple and $B$ gives me a pear, I have an apple and a pear'). On the other hand, information is not a resource and can be copied ad libitum: I can tell everyone that I have an apple as long as I want.

To aid intuition, we highlight the difference between exchanges of resources and exchanges of information by depicting the former in green and the latter in red.

Countless categorical structures have been proposed to achieve a complete formalisation of this difference. Here, we will try to get away with as few assumptions as we can.

In this setting, a simple trade like the one in~\autoref{fig: simple trade} can be represented as a couple of morphisms
\begin{equation*}
  \begin{tikzpicture}
    \xScale[2]{\gau{$A$}\mor[green!10, minimum height=12, minimum width=12]{$\faBtc$}\dro{$B$}}
  \end{tikzpicture}
  \hspace{8em}
  \begin{tikzpicture}
    \xScale[2]{\gau{$B$}\mor[green!10, minimum height=12, minimum width=12]{$\faEnvelope$}\dro{$A$}}
  \end{tikzpicture}
\end{equation*}
oriented from left to right.

A morphism $A \to B$ represents a money transfer from $A$ to $B$ (in this case, Bitcoins). Similarly, a morphism $B\to A$ represents the transfer of goods from $B$ to $A$.

Composing these morphisms expresses the idea of $A$ turning her money into goods and $B$ turning his goods into money.
An escrow, as we designed it, is given by:
\begin{equation}\label{eq: escrow}
  \begin{tikzpicture}
    \yScale[2]{\gau{$A$}\arTwo[green!10]{$\faLock$}}
    \down[2]{
      \yScale{
        \step{
          \node[below, font=\tiny] at (0,2*\nlen) {$B$};
          \mor[red!10, minimum height=12, minimum width=12]{$\faEnvelope$}
          \step{
            \node[below, font=\tiny] at (0,2*\nlen) {$A$};
          }
          \up[2]{
            \idn
            \step[0.5]{\node[below, font=\tiny] at (0,2*\nlen) {$E$};}
          }
        }
      }
    }
    \step[2]{
      \yScale[2]{
        \twoAr[green!10]{$\faUnlock$}\dro{$B$}
      }
    }
  \end{tikzpicture}
\end{equation}

This is interpreted as follows:
\begin{itemize}
  \item the morphism \texttt{Lock} takes a resource owned by $A$ and transfers it to a trusted escrow $E$. This can be a third trusted party or a digital vault such as a smart contract. It also sends to $B$ a notification that the resource has been deposited, so that shipment can begin.
  \item The morphism \texttt{Unlock} transfers the resource from $E$ to $B$ upon confirmation from $A$.
  \item The morphsim \texttt{Envelope} represents the physical shipment and reception of goods from $B$ to $A$.
\end{itemize}

An important remark on the graphical representation of optics is now in order:
\begin{remark}\label{our_naive_optics}
  Our treatment of optics is based on a less sophisticated idea than the one used in functional programming, i.e. the intuition that an optic is a device that can be pictorially represented as a block
  \[
    \begin{tikzpicture}
      \fill[lightgray] (0,0) rectangle (2,1);
      \fill[white] (.25,-.1) rectangle (1.75,.5);
      \node[left] at (0,.5) {\tiny $S$};
      \node[right] at (2,.5) {\tiny $T$};
      \draw[wire] (.25,.25) -- (.5,.25) node[right, font=\tiny] {$A$};
      \draw[wire, xshift=1.25cm] (.25,.25) -- (.5,.25) node[xshift=-.25cm, left, font=\tiny] {$B$};
    \end{tikzpicture}
  \]
  `waiting for' a morphism $f : A\to B$ to produce a morphism $S \to T$.
\end{remark}
Notice how the morphism $B \to A$ above is depicted in red. It represents the \emph{information} witnessing that an exchange of goods has happened.

This is compatible with a setting --such as the blockchain-- where the morphism $A \to B$ (exchange of money) can be accounted for and publicly verified, while the morphism $B \to A$ (the goods have been shipped and received), cannot.

We see that, in this perspective, an escrow is nothing more than a way to wrap an unaccountable morphism $B \to A$ inside an accountable morphism $A \to B$: In other words, we split the morphism $A \to B$ in two parts, and sandwich the morphism $B \to A$ in the middle. Taking things apart, we get:
\begin{equation*}
  \begin{tikzpicture}
    \yScale[2]{\gau{$A$}\arTwo[green!10]{$\faLock$}}
    \down[2]{
      \yScale{
        \step{
          \node[below, font=\tiny] at (0,2*\nlen) {$B$};
          \step{
            \node[below, font=\tiny] at (0,2*\nlen) {$A$};
          }
          \up[2]{
            \idn
            \step[0.5]{\node[below, font=\tiny] at (0,2*\nlen) {$E$};}
          }
        }
      }
    }
    \step[2]{
      \yScale[2]{
        \twoAr[green!10]{$\faUnlock$}\dro{$B$}
      }
    }
  \end{tikzpicture}
  \hspace{8em}
  \begin{tikzpicture}
    \gau{$B$}\mor[red!10, minimum height=12, minimum width=12]{$\faEnvelope$}\dro{$A$}
  \end{tikzpicture}
\end{equation*}
The diagram on the left now has the familiar shape of a comb, furthering our suspects that it is indeed an optic, while the diagram on the right is just a morphism. Recall the following, which is a particular case of~\cite[2.1]{2001.07488}.
\begin{definition}[The category of optics]
  Let $\clM$ be a monoidal category; define the category $\Opt_\clM$ (or $\Opt$ for short) of optics as having
  \begin{itemize}
    \item objects the pairs $\cop{A}{B}$ in $\clM$;
    \item morphisms $\cop A B \to \cop S T$ defined by the coend
          \[\label{panopticon}\int^{M\in\clM} \clM(S,M\otimes A)\times \clM(M\otimes B,T).\]
  \end{itemize}
  Each such morphism is called an \emph{optic} with domain $\cop A B$ and codomain $\cop S T$.
\end{definition}
\begin{remark}
  With this definition in mind, we can provide an equational description of the escrow diagram in \cref{eq: escrow}. The diagram therein should clarify the aforementioned intuition that an equivalence class in the coend in Equation \eqref{panopticon} consists of a pair of morphisms $u : S \to M\otimes A $ and $v : M\otimes B \to T$ `waiting for' a morphism $f : A\to B$ in order to compose into a morphism $S \to T$ as
  \[\xymatrix{S \ar[r]^-u & M\otimes A \ar[r]^{M\otimes f} & M \otimes B \ar[r]^-v & T}\]
  Evidently, the family of maps $\tau_M : \clM(S,M\otimes A)\times \clM(M\otimes B,T) \times \clM(A,B)$ determined by the rule $(u,v,f)\mapsto v \circ M\otimes f \circ u$ is a co-wedge (cf. \cite[1.1.4]{coend-calcu}) in $M$, so that there exists a unique $\bar\tau  : \Opt(\cop A B, \cop S T) \to \clM(S,T)$.

  Graphical manipulation provides an efficient way to handle the coend and the equivalence classes therein.
\end{remark}
We are now ready to introduce our main definition:
\begin{definition}
  An \emph{escrow} from $A$ to $B$ is an optic $\cop BA \to \cop AB$.
\end{definition}
We record another definition for further use, that of a \emph{Tambara module}: a symmetric version of the original definition in \cite[§3]{pastro2008doubles} is enough for our purposes; a more general notion of Tambara module for a pair left/right actegory on a monoidal base is given in \cite[4.1]{2001.07488}.
\begin{definition}\label{tambara}
  Let $\clM$ be a (symmetric) monoidal category; a \emph{Tambara module} consists of a profunctor $P : \clM \pto \clM$ endowed with maps
  \[\xymatrix{\alpha^{[M]}_{AB} : P(A,B) \ar[r] & P(M\otimes A,M\otimes B)}\]
  satisfying suitable compatibility conditions with the associator and the unitor of the monoidal structure:
  \begin{itemize}
    \item the diagram
          \[\vcenter{\xymatrix{
              P(A,B) \ar[r]^-\alpha\ar[d]_\alpha& P((M\otimes N)\otimes A,M\otimes N\otimes B)\ar[d]^{(\star)}\\
              P(N\otimes A, N\otimes B) \ar[r]_-\alpha & P(M\otimes (N\otimes A), M\otimes (N\otimes B))
            }}\]
          commutes, where $(\star)$ is induced by functoriality of $P$ by the associator map of $\clM$;
    \item the diagram
          \[\vcenter{\xymatrix{
            P(A,B)\ar[r]^-{\alpha}\ar@{=}[dr] & P(I\otimes A, I\otimes B) \ar[d]^{(\star\star)}\\
            & P(A,B)
            }}\]
          commutes, if $(\star\star)$ is induced by functoriality of $P$ by the unitor of $\clM$.
  \end{itemize}
\end{definition}
\begin{remark}
  Tambara modules are the objects of a category $\cate{Tamb}$ where the hom-set between a Tambara module $(P,\alpha)$ and a Tambara module $(Q,\beta)$ is the pullback of certain two equalisers (we will not need an explicit description of this hom-set; the reader is invited to consult \cite{2001.07488} for the definition).
\end{remark}
All this will turn out useful in the proof of our Proposition \ref{Vermittler}.
\section{Composing escrows}\label{sec:compescrow}
Leveraging on the notions of composition for optics described in \cite{composing_opt1} and \cite{composing_opt2}, we get that we can compose escrows as follows:
\begin{equation*}
	\begin{tikzpicture}
		\yScale[2]{\gau{$A$}\arTwo[green!10]{$\faLock$}}
		\down[2]{
			\yScale{
				\step{
					\node[below, font=\tiny] at (0,2*\nlen) {$B$};
					\step{
						\node[below, font=\tiny] at (0,2*\nlen) {$A$};
					}
					\up[2]{
						\idn
						\step[0.5]{\node[below, font=\tiny] at (0,2*\nlen) {$E_1$};}
					}
				}
			}
		}
		\step[2]{
			\yScale[2]{
				\twoAr[green!10]{$\faUnlock$}
			}
		}
		\step[3]{
			\yScale[2]{\node[below, font=\tiny] at (0,2*\nlen) {$B$};}
		}
		\step[3]{
			\yScale[2]{\arTwo[green!10]{$\faLock$}}
			\down[2]{
				\yScale{
					\step{
						\node[below, font=\tiny] at (0,2*\nlen) {$C$};
						\step{
							\node[below, font=\tiny] at (0,2*\nlen) {$B$};
						}
						\up[2]{
							\idn
							\step[0.5]{\node[below, font=\tiny] at (0,2*\nlen) {$E_2$};}
						}
					}
				}
			}
			\step[2]{
				\yScale[2]{
					\twoAr[green!10]{$\faUnlock$}\dro{$C$}
				}
			}
		}
	\end{tikzpicture}
	\begin{tikzpicture}
		\up[2]{\gau{$A$}}
		\yScale{\arTwo[green!10]{$\faLock$}}
		\step{
			\node[below, font=\tiny] at (0,2*\nlen) {$B$};
			\up[4]\idn
			\arTwo[green!10]{$\faLock$}
			\down{\dro{$A$}}
			\step[0.5]{\down{\dro{$B$}}}
		}
		\step[2]{
			\up[4]{
				\idn
				\step[0.5]{\node[below, font=\tiny] at (0,2*\nlen) {$E_1$};}
			}
			\up[1]{
				\idn
				\step[0.5]{\node[below, font=\tiny] at (0,2*\nlen) {$E_2$};}
			}
		}
		\step[3]{
			\twoAr[green!10]{$\faUnlock$}
			\up[4]\idn
			\step{
				\node[below, font=\tiny] at (0,2*\nlen) {$A$};
				\yScale{\twoAr[green!10]{$\faUnlock$}\dro{$B$}}
			}
		}
	\end{tikzpicture}
\end{equation*}
Although very natural, both compositions are rather uninteresting. The one on the left models the idea that \B uses the money obtained from an escrow deal with \A to open an escrow deal with a third entity $C$ (\C).
The one on the right composes escrows with the usual optic composition and models the fact that escrows between $A$ and $B$, being either optics $\cop A B \to \cop B A$ or optics $\cop B A \to \cop A B$, can be nested into each other (indefinitely).

What we are interested in is a notion of composition that models trades happening with an intermediary: $A$ opens an escrow trade with $B$, which in turns makes an escrow trade with $C$, obtains the goods, and sends them back to $A$:
\begin{equation*}
	\begin{tikzpicture}
		\up[2]{\gau{$A$}}
		\yScale{\arTwo[green!10]{$\faLock$}}
		\step{
			\node[below, font=\tiny] at (0,2*\nlen) {$B$};
			\up[4]\idn
			\arTwo[green!10]{$\faLock$}
			\down{\dro{$C$}}
			\step[0.5]{\down{\dro{$A$}}}
			\up[4]{\node[below, font=\tiny] at (0,2*\nlen) {$E_1$};}
		}
		\step[2]{
			\up[4]{
				\down[2.5]{\yScale[1.5]{\braid}}
			}
		}
		\step[3]{
			\twoAr[green!10]{$\faUnlock$}
			\up[4]\idn
			\step{
				\node[below, font=\tiny] at (0,2*\nlen) {$B$};
				\yScale{\twoAr[green!10]{$\faUnlock$}\dro{$C$}}
				\up[4]{\node[below, font=\tiny] at (0,2*\nlen) {$E_2$};}
			}
		}
	\end{tikzpicture}
	\hspace{8em}
	\begin{tikzpicture}
		\gau{$C$}\mor[red!10, minimum height=12, minimum width=12]{$\faEnvelope$}
		\step{\node[below, font=\tiny] at (0,2*\nlen) {$B$};}
		\step{\mor[red!10, minimum height=12, minimum width=12]{$\faEnvelope$}\dro{$A$}}
	\end{tikzpicture}
\end{equation*}
As one can see, there is a witness $B \to A$ that unlocks the escrow that $A$ makes with $B$, and a witness $C \to B$ that unlocks the escrow that $B$ makes with $C$. These two witnesses are composed to obtain the witness that unlocks the composite escrow. This makes intuitive sense: $B$ should not be able to unlock his funds until the delivery of goods from $C$ all the way to $A$ is completed. On the other hand, $C$ depends on $B$'s behavior for their own payment. To better understand this, imagine the following scenario: $C$ successfully delivers the goods to $B$, while $B$ fails to deliver to $A$. In such a setting, $C$ won't be paid even if they completed their work successfully. This is because the escrow composition defers $C$'s payment to the moment when the whole escrow sequence is completed. Moreover, we see that in such a setting $B$ unlocks his own escrow before $C$, and as such we informally deem $B$'s behavior as `greedy'. There are other ways to compose the building blocks of these two escrows to obtain a `gentleman $B$' scenario where the escrow between $B$ and $C$ is settled indepentently --and before-- of $B$'s escrow with $A$, but we weren't able to find a convincing categorical description of this composition, which we thus omit.

\begin{figure}
	\begin{center}
		\begin{tikzpicture}[>=stealth']
			\draw[->] (0,0) -- (0,-6);
			\draw[->] (2.5,0) -- (2.5,-6);
			\draw[->] (5,0) -- (5,-6);
			\draw[->] (7.5,0) -- (7.5,-6);
			\draw[->] (10,0) -- (10,-6);
			\node[nd] (A1)  at (0,-1){\aki};
			\node[nd] (A2) at (0,-5) {\aki};
			\node[nd] (V1) at (2.5,-1.5) {\vault};
			\node[nd] (B1) at (5,-1.5) {\bogdan};
			\node[nd] (V2) at (2.5,-2.5) {\vault};
			\node[nd] (B2) at (5,-2.5) {\bogdan};
			\node[nd] (B3) at (5,-4) {\bogdan};
			\node[nd] (V'1) at (7.5,-3) {\vault};
			\node[nd] (C1) at (10,-1.5) {\candice};
			\node[nd] (C2) at (10,-3) {\candice};
			\node[nd] (V3) at (2.5,-5) {\vault};
			\node[nd] (B4) at (5,-5) {\bogdan};
			\node[nd] (V'2) at (7.5,-5) {\vault};
			\node[nd] (C3) at (10,-5) {\candice};
			\draw[->, thick, gray] (A1) -- node[lbl] {lock \$} (V1);
			\draw[->, thick, gray] (B1) -- node[lbl] {locked?} (V2);
			\draw[->, thick, gray] (V2) -- node[lbl] {yes} (B2);
			\draw[->, thick, gray] (B2) -- node[lbl] {lock \$} (V'1);
			\draw[->, thick, gray] (C1) -- node[lbl] {locked?} (V'1);
			\draw[->, thick, gray] (V'1) -- node[lbl] {yes} (C2);
			\draw[->, thick, gray] (C2) -- node[lbl] {send goods} (B3);
			\draw[->, thick, gray] (B3) -- node[lbl] {send goods} (A2);
			\draw[->, thick, gray] (A2) -- node[lbl] {unlock} (V3);
			\draw[->, thick, gray] (V3) -- node[lbl] {send \$} (B4);
			\draw[->, thick, gray] (B4) -- node[lbl] {unlock} (V'2);
			\draw[->, thick, gray] (V'2) -- node[lbl] {send \$} (C3);
		\end{tikzpicture}
		\caption{An escrow trade with a greedy intermediary.}\label{fig: escrow_greed_trade}
	\end{center}
\end{figure}

We now formalize the escrow composition highlighted in the picture above. In the following, we shorten the set $\Opt(\cop S T,\cop T S)$ as $\lr T S$.
\begin{definition}[The building blocks of escrow composition]
	For a pair $(h,k)$ of elements in $\lr UT\times\lr TS$, we define a `composition'
\[\xymatrix{\firstblank\diamond\firstblank : \lr UT\times\lr TS \ar[r] & \lr US} : (k,h)\mapsto k\diamond h\]
as follows:
\begin{itemize}
	\item first choose representative elements for $h,k$ inside the equivalence classes of the two coends involved: let's say $h \in \int^{M\in\clM} \clM(T,M\otimes S)\times \clM(M\otimes T,S)$ is represented by a pair $(h_1,h_2)$ where
	      \[h_1 : T \to M\otimes S, \quad h_2 : M\otimes T \to S\]
	      for $M\in\clM$, and similarly, $k \in \int^{N\in\clM} \clM(U,N\otimes T)\times \clM(N\otimes U,T)$ is a pair $(k_1,k_2)$:
	      \[k_1 : U\to N\otimes T, \quad k_2 : N\otimes U \to T\]
	      for $N\in\clM$.
	\item second, send these pairs to the composite arrows
	      \[\Big((N\otimes h_1)\circ k_1 , h_2 \circ (M\otimes k_2) \circ (\sigma_{N,M} \otimes U) \Big)\]
\end{itemize}
In terms of string diagrams, this is what happens: the two pairs
\begin{gather}
	\left(
	\begin{tikzpicture}[baseline=(current bounding box.center)]
		\arTwo[magenta!10]{h_1}
	\end{tikzpicture}\, , \,
	\begin{tikzpicture}[baseline=(current bounding box.center)]
		\twoAr[blue!10]{h_2}
	\end{tikzpicture}
	\right)
	\qquad
	\left(  \begin{tikzpicture}[baseline=(current bounding box.center)]
		\arTwo[teal!10]{k_1}
	\end{tikzpicture}\, , \,
	\begin{tikzpicture}[baseline=(current bounding box.center)]
		\twoAr[yellow!10]{k_2}
	\end{tikzpicture} \right)
\end{gather}
get sent to  the (equivalence class of the) pair
\[\label{qui}
	\left(
	\begin{tikzpicture}[baseline=(current bounding box.center)]
			\up[2]{\gau{$U$}}
			\yScale{\arTwo[teal!10]{k_1}}
			\step{\arTwo[magenta!10]{h_1}\up[4]\idn
				\down{\dro{$S$}}
				\up{\dro{$M$}}
				\up[4]{\dro{$N$}}
			}
		\end{tikzpicture}
	\quad ,
	\quad
	\begin{tikzpicture}[baseline=(current bounding box.center)]
			\down{\gau{$U$}}
			\up{\gau{$M$}}
			\up[4]{\gau{$N$}}
			\up[1.5]{\yScale[1.5]\braid}\down\idn
			\step{\twoAr[yellow!10]{k_2}\up[4]\idn}
			\step[2]{\yScale{\twoAr[blue!10]{h_2}}\up[2]{\dro{$S$}}}
		\end{tikzpicture}
	\right)
\]
where a twisting $N\otimes M \to M\otimes N$ has been introduced in order for $(k\diamond h)_1$ and $(k\diamond h)_2$ to have the correct domains.
\end{definition}
\begin{remark}
	This is a well-defined correspondence, i.e. it yields the same result if we change representatives in the same equivalence classes of $h,k$: in order to see this, recall that the equivalence relation in the coend in Equation \eqref{panopticon} is generated by the pair of maps
	\[\xymatrix{\displaystyle\sum_{u : M\to N} \clM(T,M\otimes S)\times \clM(N\otimes T, S) \ar@<.33em>[r]^-l\ar@<-.33em>[r]_-r & \displaystyle\sum_A\, \clM(T,A\otimes S)\times \clM(A\otimes T,S)}\]
	acting as follows for a given $u : M \to N$:
	\begin{itemize}
		\item $l_u(f,g) = \big((u\otimes S)\circ f, g\big)$;
		\item $r_u(f,g) = \big(f, g\circ (u\otimes T)\big)$.
	\end{itemize}
	The quickest way to see that the operation defined in Equation \eqref{qui} is compatible with this equivalence relation is through string diagram representation: writing morphisms as `combs' we get that
	\begin{align*}
		l_{(u,v)}(h_1,h_2)\diamond (k_1,k_2) & = \left(
		\begin{tikzpicture}[scale=.5, baseline=(current bounding box.center)]
				\yScale{\arTwo[teal!10]{}}
				\step{\arTwo[magenta!10]{}\up[4]\idn
				}
				\yScale[1.5]{\up{\step[2]{\UDmor{u}{v}}}}\step[2]{\Did}
			\end{tikzpicture}
		\quad ,
		\quad
		\begin{tikzpicture}[scale=.5, baseline=(current bounding box.center)]
				\up[1.5]{\yScale[1.5]\braid}\down\idn
				\step{\twoAr[yellow!10]{}\up[4]\idn}
				\step[2]{\yScale{\twoAr[blue!10]{}}
				}
			\end{tikzpicture}
		\right)                                                                         \\
		                                     & =
		\left(
		\begin{tikzpicture}[scale=.5, baseline=(current bounding box.center)]
				\yScale{\arTwo[teal!10]{}}
				\step{\arTwo[magenta!10]{}\up[4]\idn
				}
			\end{tikzpicture}
		\quad ,
		\quad
		\begin{tikzpicture}[scale=.5, baseline=(current bounding box.center)]
				\step[-1]{\yScale[1.5]{\up{\UDmor{v}{u}}}\Did}
				\up[1.5]{\yScale[1.5]\braid}\down\idn
				\step{\twoAr[yellow!10]{}\up[4]\idn}
				\step[2]{\yScale{\twoAr[blue!10]{}}
				}
			\end{tikzpicture}
		\right)                                                                         \\
		                                     & = (h_1,h_2) \diamond r_{(u,v)}(k_1,k_2).
	\end{align*}
\end{remark}
\begin{proposition}
	The rule above is associative and unital: this means that the diagrams
	\begin{adju}
		\xymatrix{
		\lr VU\times \lr U T \times \lr TS\ar[r]^-{\lr V U \times \diamond}\ar[d]_{\diamond\times\lr T S} & \lr VU \times \lr US \ar[d]^\diamond\\
		\lr VT \times \lr TS \ar[r]_-\diamond & \lr V S
		}
		\quad
		\xymatrix{
		\lr T S \ar[r]^-{\lr TS\times\nu_S}\ar[d]_{\nu_T \times\lr TS } & \lr T S \times \lr S S \ar[d]^\diamond \\
		\lr T T \times \lr T S \ar[r]_-\diamond & \lr T S
		}
	\end{adju}
	commute for a choice of `unit' $\nu_X : 1 \to \lr X X$.
\end{proposition}
\begin{proof}
	Associativity: suppose given a triple $(h,k,p)$ such that
	\[
		h = \left(
		\begin{tikzpicture}[baseline=(current bounding box.center)]
				\arTwo[magenta!10]{h_1}
			\end{tikzpicture}\, , \,
		\begin{tikzpicture}[baseline=(current bounding box.center)]
				\twoAr[blue!10]{h_2}
			\end{tikzpicture}
		\right)\quad
		k =   \left(  \begin{tikzpicture}[baseline=(current bounding box.center)]
				\arTwo[teal!10]{k_1}
			\end{tikzpicture}\, , \,
		\begin{tikzpicture}[baseline=(current bounding box.center)]
				\twoAr[yellow!10]{k_2}
			\end{tikzpicture} \right)\quad
		p =   \left(
		\begin{tikzpicture}[baseline=(current bounding box.center)]
				\arTwo[orange!10]{p_1}
			\end{tikzpicture}\, , \,
		\begin{tikzpicture}[baseline=(current bounding box.center)]
				\twoAr[cyan!10]{p_2}
			\end{tikzpicture}
		\right)
	\]
	In string diagram notation, the two possible compositions of $h,k,p$ are given, respectively, by the following pairs:
	\[
		\underset{(h\diamond k)\diamond p}{\left( \begin{tikzpicture}[scale=.5,baseline=(current bounding box.center)]
				\yScale[3]{\arTwo[orange!10]{}}
				\down{\yScale{\step{\arTwo[teal!10]{}\xScale{\up[3]\idn}}}
					\step[2]{\arTwo[magenta!10]{}\up[4]\idn}}
			\end{tikzpicture}
			\,,\,
			\begin{tikzpicture}[scale=.75]
				\step[-1]{\up[3.5]\twoid}
				\down[.5]{\step[-1]{\yScale[3]\coturn}
					\twoAr[cyan!10]{}}
				\up[3.5]{\braid}
				\step{\yScale[1.5]{\twoAr[yellow!10]{}}}
				\step[2]{\up[1.25]{\yScale[1.75]{\twoAr[blue!10]{}}}}
				\step{\up[4.5]\idn}
				\step[-1]{\down[1.5]\idn}
			\end{tikzpicture}
			\right)}
		\qquad
		\underset{h\diamond (k\diamond p)}{\left(
			\begin{tikzpicture}[scale=.5,baseline=(current bounding box.center)]
				\yScale[3]{\arTwo[orange!10]{}}
				\down{\yScale{\step{\arTwo[teal!10]{}\xScale{\up[3]\idn}}}
					\step[2]{\arTwo[magenta!10]{}\up[4]\idn}}
			\end{tikzpicture}
			\,,\,\begin{tikzpicture}[scale=.75]
				\down[.5]{\up[3]\idn\twoAr[cyan!10]{}}
				\step{\yScale[1.5]{\twoAr[yellow!10]{}}}
				\step[2]{\up[1.25]{\yScale[1.75]{\twoAr[blue!10]{}}}}
				\step{\up[4.5]\idn}
				\step[-1]{\up[1.5]{\braid}\down[1.5]\idn}
				\up[1.5]{\step[-1]{\draw[wire] (0,0) -| (1,1.25) -- ++(1,0);}}
			\end{tikzpicture}
			\right)}
	\]
	Unitality: define $\nu_S$, as the equivalence class of the pair $(\lambda^{-1},\rho)$ in the coend
	\[\label{aoaou}\int^M \clM(S,M\otimes S)\times \clM(M\otimes S,S)\]
	(more formally, consider the summand at $M=I$ --the monoidal unit; the product of hom sets
	\[\clM(S,I\otimes S)\times \clM(I\otimes S,S)\]
	contains the element $\left(\var{S}{I\otimes S},\var{I\otimes S}{S}\right)$ obtained from the unitors of the monoidal structure; now project on the quotient that defines the coend in Equation \eqref{aoaou}) and use a straightforward string diagrammatic argument.
\end{proof}
\section{Using agents as intermediaries}\label{sec:usingAgentsAsIntermediaries}
In Section \ref{sec:compescrow} we described escrow composition roughly as `to trade with \C, \A opens an escrow with \B, \B opens an escrow with \C, goods are delivered, and the escrows are separately settled'. This kind of composition has its drawbacks, and to understand them, we will have to briefly venture into the exciting (?) world of legal matters.

In our escrow composition, $B$ is essentially buying the goods from $C$ on behalf of $A$. Legally, the trades between $A$ and $B$ and $B$ and $C$ are independent, and $B$ is considered a freelancer, for which the law entails a precise taxation regime.

Nevertheless, sometimes we would like to have $B$ considered as an employee working for $C$. For example, $C$ may be a producer, and $B$ may be an intermediate agent that stockpiles a warehouse of goods produced by $C$, selling them on their behalf. Tax-wise, these two scenarios are very different, so we want to design an escrow procedure `sensitive to the taxation regime'. One possibility is the following:
\begin{equation*}
  \begin{tikzpicture}
    \yScale[1]{\gau{$A$}\arTwo[green!10]{$\faLock$}}
    \down[1]{
      \step{
        \node[below, font=\tiny] at (0,2*\nlen) {$C$};
        \step{
          \node[below, font=\tiny] at (0,2*\nlen) {$A$};
        }
        \up[2]{
          \idn
          \step[0.5]{\node[below, font=\tiny] at (0,2*\nlen) {$E_A$};}
        }
      }
    }
    \step[2]{
      \yScale[1]{
        \twoAr[green!10]{$\faUnlock$}\dro{$C$}
      }
    }
    \down[4]{
      \yScale[1]{\gau{$B$}\arTwo[green!10]{$\faLock$}}
      \down[1]{
        \step{
          \node[below, font=\tiny] at (0,2*\nlen) {$C$};
          \step{
            \node[below, font=\tiny] at (0,2*\nlen) {$A$};
          }
          \up[2]{
            \idn
            \step[0.5]{\node[below, font=\tiny] at (0,2*\nlen) {$E_B$};}
          }
        }
      }
      \step[2]{
        \yScale[1]{
          \twoAr[green!10]{$\faUnlock$}\dro{$B$}
        }
      }
    }
  \end{tikzpicture}
  \hspace{8em}
  \hspace{6em}
  \begin{tikzpicture}
    \down[0]{
      \gau{$C$}\mor[red!10, minimum height=12, minimum width=12]{$\faEnvelope$}
      \step{\node[below, font=\tiny] at (0,2*\nlen) {$B$};}
      \step{\mor[red!10, minimum height=12, minimum width=12]{$\faEnvelope$}\dro{$A$}}
    }
  \end{tikzpicture}
\end{equation*}
As one can see, we have a standard escrow between $A$ and $C$. Independently from this, we have another \emph{optic} where $B$ puts some funds in escrow, but this time it is $A$ that can unlock it. In doing so, $B$ gets back his own money. In doing so, $B$ \emph{is never buying merchandise from} $C$ since there is no transfer of money from $B$ to $C$. We interpret this as having $B$ acting as an \emph{intermediary}: $B$ stakes some collateral to ensure that he will not run away with $C$'s merchandise and receives the collateral back once the trade between $A$ and $C$ is settled. 
There are a couple of striking features in this construction:
\begin{itemize}
  \item The collateralization of $B$'s money is not an optic $\cop{B}{X} \to \cop{X}{B}$ for some $X$, and thus not an escrow according to our definition;
  \item The morphism witnessing an exchange of goods from $C$ to $A$ is again obtained composing the witnesses $B \to A$ and $C \to B$ and should be used twice in this construction to fill both combs. This situation is compatible with the interpretation that combs are to be filled with red, information-like morphisms.
  \item We need to pre-compose the comb with some morphism of type $B \to A \otimes B$: intuitively, $B$ must be available before the escrow deal between $A$ and $C$ is set up, and signal his availability to $A$;
  \item We need to post-compose the comb with some morphism of type $C \otimes B \to B$: Intuitively, if $B$ works for $C$, some form of settlement between the two must take place after the mediated escrow is concluded.
\end{itemize}
\subsection{Information and resources}
Let's start by recalling a general fact: let $\clM$ be a monoidal category, let $A$ be a comonoid and $C$ be a monoid in $\clM$. Then, the hom-set $\clM(A,C)$ becomes a monoid, whose operation is called \emph{convolution}, and it is defined as
\[(f,g)\mapsto f * g : \xymatrix{A \ar[r]^-c& A\otimes A \ar[r]^{f\otimes g}& C\otimes C \ar[r]^-m& C}\]
where $c : A \to A\otimes A$ is the comultiplication and $m : C\otimes C \to C$ the multiplication map.

We can use a similar idea to prove that the set of \emph{optics} $\cop CA \to \cop AC$ is a monoid.
\setcounter{subsection}{0}
\begin{lemma}\label{itsa_monoid}
  Let $A, C$ be objects of $\clM$, let $A$ be a comonoid, and $C$ be a monoid. Then, the set of escrows $\lr A C$ is a monoid.
\end{lemma}
\begin{proof}
  One routinely checks the associativity and unit axioms from the following definitions:
  \begin{itemize}
    \item given two escrows $(h_1,h_2)$ and $(k_1,k_2)$ as in
          \[
            \esc{A}{h}{C} \qquad \esc[teal][yellow]{A}{k}{C}[N]
          \]
          we compose them into the escrow $h\otimes k$
          \[
            \begin{tikzpicture}
              \yScale{\comult}
              \step{\up[4]{\arTwo[magenta!10]{h_1}}\arTwo[teal!10]{k_1}}
              \step[2]{\up[4]{\xScale[3]\Uid}\up[2]\turn\down{\yScale{\mult}}}
              \step[3]{\up[3]\idn}
              \step[4]{
                \down{\yScale{\comult}\up[3]\coturn}
                \step{\up[4]{\twoAr[blue!10]{h_2}}\twoAr[yellow!10]{k_2}}
                \step[2]{\yScale{\mult}}
              }
            \end{tikzpicture}
          \]
          using the monoid and comonoid structures of $C, A$ respectively;
    \item the unit escrow $(i_1,i_2)$ is the lower morphism in
          \[
            \begin{tikzpicture}
              \down[2.5]{\yScale[1.5]{\step[-1]{\comult}}}
              \arTwo[magenta!10]{h_1} \down[3]{\counit \down[2]\unit}\step[4]{\twoAr[blue!10]{h_2} \down[3]{\unit \down[2]\counit}}
              \down[5]{\step{\yScale{\mult}} \step[3]{\yScale{\comult}}}
              \down[2.5]{\yScale[1.5]{\step[5]{\mult}}}
              \step{\up{\xScale[3]\idn}}
              \down[4]\akasa \step[4]{\down[4]\akasa}
              \node at (.5,-.5) {\tiny $i_1$};
              \node at (4.5,-.5) {\tiny $i_2$};
            \end{tikzpicture}
          \]
          obtained using the counit of $A$ and the unit of $C$. Evidently, the co/monoid rules entail that $h\otimes i$ is just equal to $h$, for every other optic $h=(h_1,h_2)$. Similarly, one sees that $i\otimes k=k$. \qedhere
  \end{itemize}
\end{proof}
\begin{remark}
  It is worth spelling out the precise way in which we obtained the unit escrow $i=(i_1,i_2)$: its components consist of the equivalence class of $(i_1,i_2)$ in the coend $\lr A C$, where we define
  \[
    \vcenter{\xymatrix{
    A \ar[d] \ar@{}[dr]|{i_1}& I\otimes C & I\otimes A \ar[d] \ar@{}[dr]|{i_2}& C \\
    I \ar[r] & I\otimes I \ar[u]& I \otimes I \ar[r]& I\ar[u]
    }}
  \]
  using the counit $\epsilon: A \to I$ and the unit $\eta: I \to C$ of the comonoid $A$ and monoid $C$.
\end{remark}
\begin{remark}
  We recall what is the structure of a \emph{module} over a monoid in $\clM$ and of a \emph{comodule} over a comonoid: a module $B$ over the monoid $C$ is an object of $\clM$ endowed with an action map $a: C\otimes B \to B$ subject to the usual axioms of compatibility with the unit and multiplication of $C$:
  \[
    \begin{tikzpicture}[yscale=1.5, xscale=.75]
      \up\unit \down\idn
      \step{\act{$a$}}
      \eql{2.5}
      \step[3]{\idn}
    \end{tikzpicture}
    \qquad \qquad
    \begin{tikzpicture}[yscale=1.5, xscale=.75]
      \up\mult
      \down\idn
      \step{\act{$a$}}
      \step[3]{
        \up\idn
        \down{\act{$a$}}
        \step{\act{$a$}}
      }
      \eql{2.5}
    \end{tikzpicture}
  \]
  expressed by the commutative diagrams
  \[
    \vcenter{\xymatrix{
      I\otimes B\ar[r]^{u\otimes B}\ar[dr]_{\lambda_B} & C \otimes B \ar[d]^a \\
      & B
    }}
    \qquad
    \vcenter{\xymatrix{
      C \otimes C \otimes B \ar[r]^{C\otimes a}\ar[d]_{m\otimes B} & C \otimes B \ar[d]^a \\
      C \otimes B \ar[r]_a & B
    }}
  \]
  using the multiplication and unit of $C$.

  Dually, an $A$-comodule is an object $B\in\clM$ endowed with a coaction map $c : B \to B\otimes A$, subject to the dual axioms of compatibility:
  \[
    \begin{tikzpicture}[yscale=1.5, xscale=.75]
      \coact{$c$}
      \step{\up\counit \down\idn}
      \eql{2.5}
      \step[3]{\idn}
    \end{tikzpicture}
    \qquad \qquad
    \begin{tikzpicture}[yscale=1.5, xscale=.75]
      \down{\coact{$c$}}
      \step{\coact{$c$}\down[2]\idn}
      \eql{2.5}
      \step[3]{
        \coact{$c$}
        \step{\up\idn}
        \step{\down\comult}
      }
    \end{tikzpicture}
  \]
  expressed by the commutative diagrams
  \[
    \vcenter{\xymatrix{
      B \ar[dr]^{\rho_B}\ar[d]_c & \\
      B\otimes A \ar[r]_-{B\otimes \epsilon} & B\otimes I
    }}
    \qquad
    \vcenter{\xymatrix{
      B \ar[r]^c\ar[d]_c & B \otimes A \ar[d]^{B\otimes\sigma}\\
      B\otimes A \ar[r]_-{c\otimes A} & B\otimes A \otimes A
    }}
  \]
  using the comultiplication and counit of $A$.
\end{remark}
\begin{notation}\label{notat_coact}
  Let $C$ be a monoid in $\clM$ and let $B$ be a $C$-module; the action $a  : C\otimes B \to B$ is usually denoted as an infix dot, $a(x,b)=x.b$ for every $x\in C$, $b\in B$.

  Dually, let $B$ be an $A$-comodule; the coaction $c : B \to B\otimes A$ here shall be interpreted as a map sending an element $b\in B$ to a pair $(b',y)$ of an element $y\in A$ and an element $b'=b'_y$ in a new \emph{state} depending on $y$.

  Until the end of the section, we stipulate that $B$ is simultaneously a module for the monoid $C$ and a comodule for the comonoid $A$.
\end{notation}
\begin{theorem}\label{escaction}
  There exists an action
  \[\_\rtimes\_ : \lr A C \times \Opt\Big(\cop B B , \cop AC\Big) \to \Opt\Big(\cop B B , \cop AC\Big)\]
  of the monoid $\lr A C$ of Lemma \ref{itsa_monoid} on the set of optics with common codomain $\cop AC$.
\end{theorem}
\begin{proof}
  Suppose given an escrow and an optic, as the following pairs:
  \[
    \left(
    \begin{tikzpicture}[baseline=(current bounding box.center)]
        \arTwo[magenta!10]{h_1}\gau{$A$}\down{\dro{$C$}}
      \end{tikzpicture}\, , \,
    \begin{tikzpicture}[baseline=(current bounding box.center)]
        \twoAr[blue!10]{h_2}\dro{$C$}\down{\gau{$A$}}
      \end{tikzpicture}
    \right) \qquad
    \left(
    \begin{tikzpicture}[baseline=(current bounding box.center)]
        \arTwo[teal!10]{o_1}\gau{$A$}\down{\dro{$B$}}
      \end{tikzpicture}\, , \,
    \begin{tikzpicture}[baseline=(current bounding box.center)]
        \twoAr[yellow!10]{o_2}\dro{$C$}\down{\gau{$B$}}
      \end{tikzpicture}
    \right)
  \]
  Now define the action as follows:
  \[
    \begin{tikzpicture}
      \yScale{\comult}
      \step{\up[4]{\arTwo[magenta!10]{h_1}}\arTwo[teal!10]{o_1}}
      \step[2]{\up[4]{\xScale[3]\Uid}\down{\yScale{\act{$a$}}}}
      \step[4]{
        \down{\yScale{\coact{$c$}}}
        \step{\up[4]{\twoAr[blue!10]{h_2}}\twoAr[yellow!10]{o_2}}
        \step[2]{\yScale{\mult}}
      }
      \draw[wire] (2,3*\nlen) -| ++(.15,3*\nlen) -| ++(2.75,-.75) |- ++(.15,0);
    \end{tikzpicture}
  \]
  using the monoid structure of $C$, comonoid structure of $A$, module and comodule structure of $B$. We shall now prove the axioms of an action:
  \begin{enumerate}
    \item \label{a_1} \emph{unitality} axiom: the fact that $i\rtimes o=o$ for every optic $o=(o_1,o_2)$ and for the unit escrow $i\in\lr A C $.
    \item \label{a_2} the \emph{action} axiom, saying that $(k\otimes h) \rtimes o = k  \rtimes (h \rtimes o)$ for escrows $h,k$ and an optic $o$.
  \end{enumerate}
  Axiom \ref{a_1} follows from a simple string diagram deformation, using the unit axiom of the co/monoid and co/action structures:
  \[
    \begin{tikzpicture}[xscale=.75]
      \step{\up[3]\akasa}
      \yScale{\comult}
      \step{\up[4]{\counit\down[2]\unit}\arTwo[teal!10]{o_1}}
      \step[2]{\down[.5]{\yScale[1.5]{\act{$a$}}}\up[2]\turn}
      \step[4]{
        \step{\up[3]\akasa}
        \down[.5]{\yScale[1.5]{\coact{$c$}}\up[2.5]\coturn}
        \step{\up[4]{\unit\down[2]\counit}\twoAr[yellow!10]{o_2}}
        \step[2]{\yScale{\mult}}
      }
      \step[3]{\up[3]\idn}
      \begin{scope}[xshift=7.5cm]
        \yScale{\comult}
        \step{\up[4]{\counit}\arTwo[teal!10]{o_1}}
        \step[2]{
          \step{\up[4]{\unit}\twoAr[yellow!10]{o_2}}
          \step[2]{\yScale{\mult}}
        }
        \up{\step[2]\idn}
      \end{scope}
      \begin{scope}[xshift=13cm,yshift=.5cm]
        \arTwo[teal!10]{o_1}
        \step{\up\idn}
        \step[2]{\twoAr[yellow!10]{o_2}}
      \end{scope}
    \end{tikzpicture}
  \]
  Similarly, axiom \ref{a_2} follows from the co/associativity of the co/action and co/multiplication applied to the following equality of wire diagrams, where on the left we have $k\rtimes (h\rtimes o)$ (slightly simplifying its structure to avoid cognitive overload), and on the right $(k\otimes h)\rtimes o$:
  \[
    \begin{tikzpicture}[xscale=.75]
      \step[-1.5]{\down[4.75]{\yScale[2.25]\comult}}
      \step[-.5]{\down[5.5]{\xScale[.5]{\yScale[1.5]\comult}}\xScale[.5]\idn}
      \arTwo[orange!10]{k_1}
      \down[3]{\arTwo[magenta!10]{h_1}}
      \down[6]{\arTwo[teal!10]{o_1}}
      \step{\up{\xScale\idn}}
      \step{\down[2]{\xScale\idn}}
      \step{\down[5]{\xScale\idn}}
      \step[3]{
        \twoAr[cyan!10]{k_2}
        \down[3]{\twoAr[blue!10]{h_2}}
        \down[6]{\twoAr[yellow!10]{o_2}}
        \step{\down[5.5]{\xScale[.5]{\yScale[1.5]\mult}}\xScale[.5]\idn}
        \step[1.5]{\down[4.75]{\yScale[2.25]\mult}}
      }
      \draw[wire] (1,-1.25) -| ++(\nlen,3*\nlen) -- ++(-\nlen,0);
      \draw[wire] (1.25,-1) -| ++(.5,5*\nlen) -- ++(-3*\nlen,0);
      \draw[wire] (3,-1.25) -| ++(-\nlen,3*\nlen) -- ++(\nlen,0);
      \draw[wire] (2.75,-1) -| ++(-.5,5*\nlen) -- ++(3*\nlen,0);
    \end{tikzpicture}
    \quad
    \begin{tikzpicture}[xscale=.75]
      \step[-1.5]{\down[6.25]{\yScale[2.25]\comult}}
      \step[-.5]{\down[2.5]{\xScale[.5]{\yScale[1.5]\comult}}\down[6]{\xScale[.5]\idn}}
      \arTwo[orange!10]{k_1}
      \down[3]{\arTwo[magenta!10]{h_1}}
      \down[6]{\arTwo[teal!10]{o_1}}
      \step{\up{\xScale\idn}}
      \step{\down[2]{\xScale\idn}}
      \step{\down[5]{\xScale\idn}}
      \step[3]{
        \twoAr[cyan!10]{k_2}
        \down[3]{\twoAr[blue!10]{h_2}}
        \down[6]{\twoAr[yellow!10]{o_2}}
        \step{\down[2.5]{\xScale[.5]{\yScale[1.5]\mult}}\down[6]{\xScale[.5]\idn}}
        \step[1.5]{\down[6.25]{\yScale[2.25]\mult}}
      }
      \up[3]{\draw[wire] (1,-1.25) -| ++(\nlen,3*\nlen) -- ++(-\nlen,0);}
      \down{\draw[wire] (1,-1) -| ++(.5,4*\nlen) -- ++(-\nlen,0);}
      \up[3]{\draw[wire] (3,-1.25) -| ++(-\nlen,3*\nlen) -- ++(\nlen,0);}
      \down{\draw[wire] (3,-1) -| ++(-.5,4*\nlen) -- ++(\nlen,0);}
    \end{tikzpicture}
  \]
  This concludes the proof.
\end{proof}
Unfortunately, one can readily notice that \cref{escaction}, albeit interesting, does not provide the notion of composition we are looking for. Roughly speaking, the most evident problem here is that our action does not even have the right type signature: $B$ is present on both sides in the inside of the resulting comb so that we do not even have a place to plug in our witnesses now, that have signature $C \to B$ and $B \to A$! To define a meaningful composition representing $B$ acting as an intermediary, we then need to resort to the underlying structure of Tambara modules that we described in \cref{tambara}.
\begin{assumption}
The blanket assumption all along the following construction remains that the object $B\in\clM$ is simultaneously a module for the monoid $C$ and a comodule for the comonoid $A$. Under this assumption, we can model the situation where $B$ acts as an intermediary in a transaction between $A$ and $C$.
\end{assumption}
\begin{definition}[Vermittler optics]
  (From German, \emph{Vermittler} = intermediary) Given $A,C$ a comonoid and a monoid in $\clM$, and $B$ an object which is at the same time a $C$-module and an $A$-comodule, a \emph{Vermittler optic} over $B$ is an element of the hom-set $\Opt(\cop CA, \cop{A\otimes B}{C\otimes B})$.
\end{definition}
We denote $\clV(B;A,C)$ the set of Vermittler optics over $B$, making the dependence on the comonoid $A$ and the monoid $C$ explicit.
\begin{definition}[$\lr CA$ `acts' on Vermittler optics]\label{Vermittler}
  Given an escrow $h$ in $\lr CA$,
  \[\left(
    \begin{tikzpicture}[baseline=(current bounding box.center)]
        \arTwo[magenta!10]{h_1}\gau{$C$}\down{\dro{$A$}}
      \end{tikzpicture}\, , \,
    \begin{tikzpicture}[baseline=(current bounding box.center)]
        \twoAr[blue!10]{h_2}\dro{$A$}\down{\gau{$C$}}
      \end{tikzpicture}
    \right)\]
  and an optic $f : \cop CA \to \cop B B $, we define a \emph{Vermittler map}
  \[\xymatrix@R=0cm{\lhd : \lr CA \times  \Opt(\cop CA, \cop BB) \ar[r] & \clV(B;A,C) \\
    (h,f)\ar@{|->}[r] & h \lhd f}\]
  between the product of the hom-set of escrows $A\to C$ and the hom-set of the Vermittler optics for $B$.
\end{definition}
\begin{theorem}
  The map in definition \ref{Vermittler} is partial, and moreover satisfies `action-like' properties:
  \begin{itemize}
    \item if $\epsilon$ is the counit of the comonoid $A$, and $\eta$ the unit of the monoid $C$, the equation
    \[\label{questa}\epsilon\circ (k\lhd f)\circ \eta  = f\]
    holds.
    \item $\lhd$ is compatible with the action and coaction maps in the sense that the square
    \[\label{laltra}\xymatrix{
      \ar[r]^{k\lhd f}& \ar[d]^{A\otimes c}\\
      \ar[u]^{C\otimes a}\ar[r]_{(n\circ f\circ m)\otimes B}&
      }\]
      is commutative, where $m : C\otimes C \to C$ is the monoid multiplication, $n : A \to A\otimes A$ the comonoid comultiplication, and $c,a$ respectively the coaction and action of $A,C$ on $B$.
    \end{itemize}
  \end{theorem}
    \begin{proof}
  Recall from \cite[4.13]{2001.07488} that the hom-set $\Opt(\cop AB, \cop ST)$ can be expressed as the end
  \[\int_P \Set(P(A,B), P(S,T))\label{dis_tamba}\]
  where $P : \clM^\op\times \clM \to \Set$ is a profunctor, and the category of integration is that  of Tambara modules \cite[4.1]{2001.07488}. Now,
  \begin{itemize}
    \item the fact that $P$ is a Tambara module yields maps
          \[\xymatrix{\alpha_{[B]}^{AC} : P(C,A) \ar[r] & P(C\otimes B, A\otimes B)}\]
          that form a wedge in $B$; post-composition with these $\alpha_{[B]}^{AC}$ and the universal wedge $\pi^\bullet$ defining Equation \eqref{dis_tamba} yields maps
          \begin{adju}
            \xymatrix{
            \lr CA \ar@{=}[r]&\int_Q \Set(Q(A,C), Q(C,A))\ar[dr]_{\tilde\alpha_P} \ar[r]^{\pi^P} & \Set(P(A,C), P(C,A)) \ar[d]^{\alpha_{[B],*}^{AC}}\\
            && \Set(P(A,C), P(C\otimes B, A\otimes B))
            }
          \end{adju}
          forming a wedge in $P$. This in turn yields a unique map
          \[\xymatrix{\tilde\alpha : \lr CA \ar[r] & \clV(B;A,C)}\]
          by the universal property of the end that defines $\clV(B;A,C)$.
    \item The assumption that $B$ is at the same time a module for the monoid $C$ and a comodule for the comonoid $A$ entails the existence of an action map $a : C\otimes B \to B$ and a coaction map $c : B \to A\otimes B$, which plugged into the profunctor $P$ yield
          \[\xymatrix{\beta_{[B]}^{AC} = P(a,c) : P(B,B) \ar[r] & P(C\otimes B, A\otimes B).}\]
          In a similar fashion, post-composition with $\beta_{[B]}^{AC}$, and the universal property of the end defining domain and codomain of $\tilde\beta$ define such a map as
          \[\tilde\beta : \xymatrix{\Opt(\cop CA, \cop BB)\ar[r] & \clV(B;A,C)}\]
          so that we obtain a cospan of maps
          \[\xymatrix{
              & \lr CA\ar[d]^{\tilde\alpha}\\
              \Opt(\cop CA, \cop BB)\ar[r]_{\tilde\beta} & \clV(B;A,C),
            }\]
          and we can consider its pullback $Z$ (in $\Set$, or inside a sufficiently cocomplete ambient category). The diagonal in the pullback is the candidate partial function
          \[
            \xymatrix{
              &Z \ar[dr]\ar[dl]& \\
              \lr CA\times \Opt(\cop CA, \cop BB) \ar@{.>}[rr]_-\lhd && \clV(B;A,C)
            }
          \]
          having $Z$ as its domain of definition. As such, $k\lhd f$ is defined as soon as $\tilde\alpha(k)=\tilde\beta(f)$, and it equals any of these two elements.
  \end{itemize}
  The main intuition behind this construction is the following: the object $Z$ is the maximal subobject of the product $\lr CA\times \Opt(\cop CA, \cop BB)$ with the property that the conjoint action and coaction over $B$, and all actions given by the Tambara module structure maps $\alpha^B$, agree. This yields a unique partial map $\firstblank \lhd \firstblank$
  that defines a \emph{partial action} of the monoid $\lr AC$ on the set of Vermittler optics over $B$.

  It only remains to show that this map satisfies the commutativities of \eqref{questa} and \eqref{laltra} when the domain of $\lhd$ has been suitably defined as above; it is certainly possible, but rather unwieldy, to argue equationally, with the definition and the universal properties of $Z$ and the hom-sets of optics. Instead, the following remark will give a graphical interpretation of the above partial action that makes the action axioms almost tautological. This will conclude the proof.
\end{proof}
\begin{remark}
  We can give a concise graphical representation of the action of the maps $\tilde\alpha,\tilde\beta$ above:
  \begin{itemize}
    \item the map $\tilde\alpha$ acts sending an escrow
          \[\left(
            \begin{tikzpicture}[baseline=(current bounding box.center)]
                \arTwo[magenta!10]{h_1}\gau{$C$}\down{\dro{$A$}}
              \end{tikzpicture}\, , \,
            \begin{tikzpicture}[baseline=(current bounding box.center)]
                \twoAr[blue!10]{h_2}\dro{$A$}\down{\gau{$C$}}
              \end{tikzpicture}
            \right)\]
          to the `trivial' Vermittler optic obtained by tensoring with the identity arrow $1_B$:
          \[\alpha(h) =\quad
            \left(
            \begin{tikzpicture}[baseline=(current bounding box.center)]
                \arTwo[magenta!10]{h_1}\gau{$C$}\down{\dro{$A$}}
                \down[3]{\idn\dro{$B$}\gau{$B$}}
              \end{tikzpicture}\, , \,
            \begin{tikzpicture}[baseline=(current bounding box.center)]
                \twoAr[blue!10]{h_2}\dro{$A$}\down{\gau{$C$}}
                \down[3]{\idn\dro{$B$}\gau{$B$}}
              \end{tikzpicture}
            \right)\]
    \item the map $\tilde\beta$ acts sending an optic $\cop CA \to \cop BB$ to the Vermittler optic obtained by plugging in the action and coaction map of $A,C$ over $B$:
          \[\tilde\beta(o) =\quad
            \left(
            \begin{tikzpicture}[xscale=.75,baseline=(current bounding box.center)]
                \act{$a$}\down{\gau{$B$}}\up{\gau{$C$}}
                \step{\arTwo[green!10]{o_1}\down{\dro{$A$}}}
              \end{tikzpicture}\, , \,
            \begin{tikzpicture}[xscale=.75,baseline=(current bounding box.center)]
                \twoAr[yellow!10]{o_2}\down{\gau{$C$}}
                \step{\coact{$c$}\up{\dro{$A$}}\down{\dro{$B$}}}
              \end{tikzpicture}
            \right)\]
  \end{itemize}
\end{remark}
The map $\lhd$ operates on the subset of the product $\lr AC \times  \Opt(\cop AC, \cop BB)$ where these two actions coincide. Intuitively, this represents the fact that the escrow mediated by $B$ (represented by $\tilde\beta(o)$) has the same effect of a standard escrow between $A$ and $C$ where $B$ just stays idle (represented by $\alpha(h)$), which signifies that, when the mediated escrow is concluded, $B$ did not receive any goods or money from any of the two parties.

To conclude, one can readily verify the `action' axioms using graphical calculus, namely the fact that representing the Vermittler optic $f$ as a pair $(f_1,f_2)$ we can follow the chain of equalities
\begin{align*}
  \begin{tikzpicture}[baseline=(current bounding box.center)]
    \arTwo[green!10]{f_1}
    \step{\up\idn}
    \step[2]{\twoAr[yellow!10]{f_2}}
    \down[3]{\act{$a$}\step\idn\step[2]{\coact{$c$}}}
  \end{tikzpicture} \kern.75em &= \kern.75em
  \begin{tikzpicture}[xscale=.75,baseline=(current bounding box.center)]
    \act{$a$}\up[2]\idn
    \up{\step{\act{$a$}}}
    \step[2]{\up{\arTwo[green!10]{f_1}}}
    \step[3]{\up[2]\idn}
    \step[4]{
      \up{\twoAr[yellow!10]{f_2}\step{\coact{$c$}}}
      \step[2]{\up[2]\idn\coact{$c$}}
    }
  \end{tikzpicture}\\[1em] &= \kern.75em
  \begin{tikzpicture}[xscale=.75,baseline=(current bounding box.center)]
    \idn\up[2]{\act{$a$}}
    \up{\step{\act{$a$}}}
    \step[2]{\up{\arTwo[green!10]{f_1}}}
    \step[3]{\up[2]\idn}
    \step[4]{
      \up{\twoAr[yellow!10]{f_2}\step{\coact{$c$}}}
      \step[2]{\idn\up[2]{\coact{$c$}}}
    }
  \end{tikzpicture}\\[1em] &= \kern.75em
  \begin{tikzpicture}[xscale=.75,baseline=(current bounding box.center)]
    \mult
    \step{\arTwo[green!10]{f_1}}
    \step[2]{\up\idn}
    \step[3]{
      \twoAr[yellow!10]{f_2}\step{\comult}
    }
    \down[3]{\idn\step\idn\step[2]\idn\step[3]\idn\step[4]\idn}
  \end{tikzpicture}
\end{align*}
This established \eqref{laltra}. Finally, that plugging the counit and unit of $A$ and $C$ at the sides of the optic $f$, we get back the original optic:
\[  \begin{tikzpicture}[xscale=.75,baseline=(current bounding box.center)]
  \step[-1]{\up\unit}
  \mult
  \step{\arTwo[green!10]{f_1}}
  \step[2]{\up\idn}
  \step[3]{
    \twoAr[yellow!10]{f_2}\step{\comult}
  }
  \step[5]{\up\counit}
\end{tikzpicture} \kern1em = \kern1em
\begin{tikzpicture}[baseline=(current bounding box.center)]
  \arTwo[green!10]{f_1}
  \step{\up\idn}
  \step[2]{
    \twoAr[yellow!10]{f_2}
  }
\end{tikzpicture}\]
This establishes \eqref{questa}.

\section{Conclusion and future work}
In this work, we defined a formalism for escrow trades based on optics. We moreover investigated various ways of chaining escrows together, which can be considered more or less convenient depending on externalities such as the tax regime the participating agents are in.

As a direction of future work, we would like to obtain an alternative formalization of the same ideas relying on new techniques such as the ones presented in \cite{nester2020structure}, that allow to model processes (such as opening an escrow vault) and passing of resources as separated interactive layers.
\putbib{bib/allofthem}{plainurl}

\end{document}